\documentclass[onecolumn]{IEEEtran}
\usepackage{amsmath,amsfonts,amssymb,euscript, graphicx, 
epsfig,
enumerate,float,afterpage, subfigure, ifthen}%

\usepackage{url}
\usepackage{hyperref}

\newtheorem{thm}{Theorem}

\newtheorem{lem}{Lemma}

\newcommand{\expect}[1]{\mathbb{E}\left[#1\right]}

\newcommand{\norm}[1]{||{#1}||}

\newcommand{\script}[1]{{{\cal{#1} }}}

\begin{document}

\title{Adaptive Utility Optimization for Stochastic Renewal Systems}

\author{Michael J. Neely \\ University of Southern California\\ \url{https://viterbi-web.usc.edu/~mjneely/}
}

%\markboth{Neely 2021}{Neely}

\maketitle

\begin{abstract} 
This paper treats a system that performs back-to-back tasks. Each new task has a random matrix of parameters and can be processed in various processing modes. The processing mode determines the task duration and a vector of rewards (penalties are viewed as negative rewards). The reward vector can include quality, profit, or $-1$ times a penalty such as 
cost or energy.   In wireless applications, tasks correspond to data transmissions and rewards/penalties  correspond to throughput and distortion. Prior work treats the problem of maximizing the time average of a reward process subject to constraints. This paper seeks to   maximize a concave function of the time average reward vector. This is difficult because it involves a multi-variable function of ratios of averages. We develop a single algorithm that does not know the probability distribution of the random task matrices. The algorithm is adaptive, meaning it responds efficiently (with optimal adaptation time) when probabilities change.
\end{abstract}

\section{Introduction}

\begin{figure}[htbp]
   \centering
   \includegraphics[height=1in]{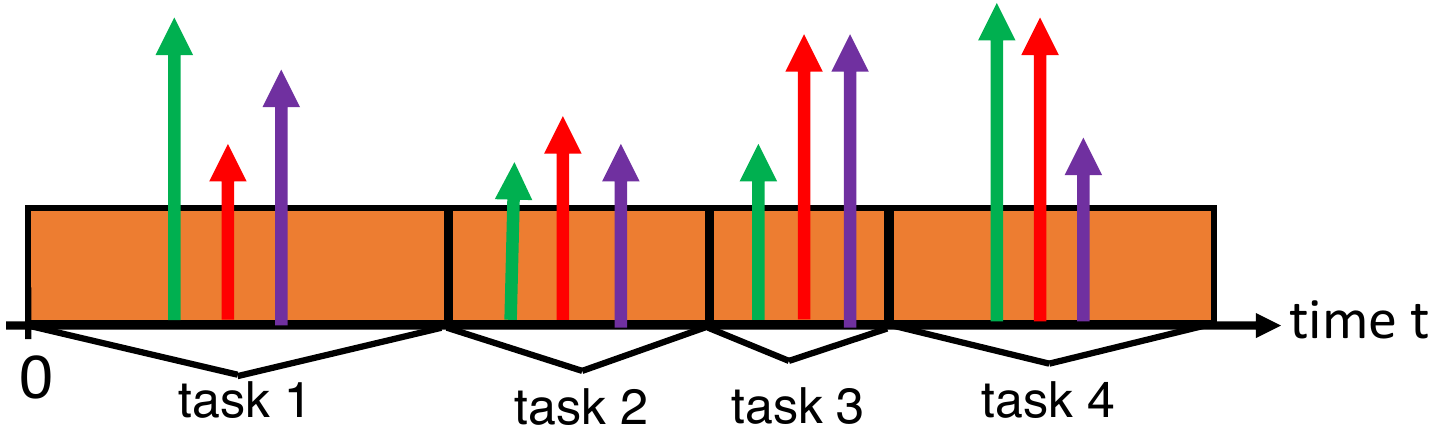} % requires the graphicx package
   \caption{An example timeline with variable-duration tasks. The green, red,  and purple arrows for each task represent profit, energy expenditure, and quality.}
   \label{fig:tasks}
\end{figure}

Fix $n$ as a positive integer. Consider a system that performs back-to-back tasks over continuous 
time $t \in [0, \infty)$. Label the tasks $k \in \{1, 2, 3, ...\}$. A new task is started once the previous task is finished (see Fig. \ref{fig:tasks}). At the start of each new task $k$, the system controller observes a random matrix $M[k]$ of parameters that describe the task.  The controller then chooses one of multiple task processing options. The option chosen, together with the task matrix $M[k]$, determine the task duration $T[k]$ (in units of time) and a vector of 
rewards $R[k]=(R_1[k], ..., R_n[k])$.  Rewards can relate to profit and quality. Negative rewards can relate to penalties such as energy expenditure, distortion, and cost. 

 Specifically, each random matrix $M[k]$ has size $S[k]\times (n+1)$, where $S[k]$ is a random positive integer. 
 The value $S[k]$ represents the number of rows of matrix $M[k]$, which is also the number of processing options 
 available for that task. Choosing a processing option for task $k$ 
 is equivalent to choosing a row of matrix $M[k]$.   Specifically, for 
 each $i \in \{1, ..., S[k]\}$, row $i$ of matrix $M[k]$ has the form 
 $$ \left[\begin{array}{ccccc} 
 T_i[k]& R_{i,1}[k] & R_{i,2}[k] & \cdots & R_{i,n}[k]
 \end{array}\right]$$
 where $T_i[k]$ is a positive-valued random variable that represents the task duration if processing option $i \in \{1, ..., S[k]\}$ is chosen; $(R_{i,1}[k], ..., R_{i,n}[k])$ is a random vector in $\mathbb{R}^n$ that represents the rewards earned under processing option $i$.  
 
\subsection{Example task matrix} 

  An example matrix for task $k$ is
 \begin{equation*}  
M[k]=\begin{array}{ccccc} 
\mbox{duration} & \mbox{profit} & \mbox{energy} & \mbox{quality} &  \\
5.1 & 3.6 & -0.3 & 2.0 & \mbox{alg 1}\\
10.2 & 2.8 & -0.7 & 3.0 & \mbox{alg 2}\\
2.7 & 3.0 & -0.2 & 1.0 & \mbox{alg 3}
\end{array}
\end{equation*} 
In this example we have $n=3$ and three rewards for profit, energy, and quality. The energy column is multiplied by $-1$ because minimizing energy is equivalent to maximizing $-1$ times energy. This example 
matrix $M[k]$ has size $3\times 4$ and there are three task processing options, one for each of the three rows. Choosing algorithm 1 (row 1) implements a processing algorithm called ``alg 1'' that yields task duration $T[k]=5.1$ and reward vector $R[k]=(3.6, -0.3, 2.0)$.   The entire matrix is observed at the start of each task $k$. Thus, the duration and reward profile for each row is known and informs the choice of row.    However, future task matrices $M[k+1], M[k+2]$, and so on, are unknown.  This type of  scheduling problem (with an observed $M[k]$) is broadly classified as \emph{opportunistic scheduling}. 
 
 \subsection{Time average reward} 
 
 Let $row[k] \in \{1, ..., S[k]\}$ denote the decision of the controller for task $k$. Then 
 \begin{align*}
 T[k] &= T_{row[k]}[k] \\
 R[k] &= (R_{row[k], 1}[k], \cdots , R_{row[k], n}[k])
 \end{align*}
 
 The time average reward vector per unit time, averaged over the first $m$ tasks, is 
 $$ \frac{\sum_{k=1}^m R[k]}{\sum_{k=1}^m T[k]} = \frac{\overline{R}[m]}{\overline{T}[m]}$$
 where we define 
 \begin{align}
 \overline{R}[m] &= \frac{1}{m}\sum_{k=1}^m R[k] \quad \forall m \in \{1, 2, 3, ...\} \label{eq:rbar}\\
 \overline{T}[m] &= \frac{1}{m}\sum_{k=1}^m T[k] \quad \forall m \in \{1, 2, 3, ...\}\label{eq:tbar}
 \end{align}

 \subsection{Optimal utility and the block i.i.d. model} 
 
 First assume the $M[k]$ matrices are i.i.d. over $k \in \{1, 2, 3, ...\}$ 
 (the entries of each $M[k]$ matrix can be correlated).  The controller  observes the components of each new $M[k]$ matrix at the start of task $k$, before making a processing mode decision for 
 that task. 
Let $\psi:\mathbb{R}^n\rightarrow \mathbb{R}$ be a concave and entrywise nondecreasing 
utility function, such as the $\alpha$-fair utility functions in \cite{mo-walrand-fair} or the proportionally-fair utility functions in \cite{kelly-charging}\cite{kelly-shadowprice}. 
The utility function can be nonsmooth.  First consider the infinite horizon problem of choosing processing modes over time to solve 
\begin{align}
&\mbox{Maximize:} \quad \liminf_{m\rightarrow\infty}\psi\left( \frac{\overline{R}[m]}{\overline{T}[m]}\right) \label{eq:p1}\\
&\mbox{Subject to:} \quad row[k] \in \{1, ..., S[k]\} \quad \forall k \in \{1, 2, 3, ....\}\label{eq:p2}
\end{align}

The problem \eqref{eq:p1}-\eqref{eq:p2} is 
a \emph{renewal optimization problem} \cite{sno-text}\cite{renewal-opt-tac}. 
Let $\psi^*$ denote the optimal utility for this problem, which 
depends on the distribution of the matrices $M[k]$.  The $\psi^*$ value optimizes over all possible algorithms for choosing a row for each task $k \in \{1, 2, 3, ...\}$, including algorithms that have 
knowledge of the probability distribution of the random matrices.  

A naive approach is to choose 
a row of $M[k]$ for each task $k$ that greedily maximizes $\psi(R[k]/T[k])$, which can be viewed as an instantaneous utility associated with task $k$.  This is generally suboptimal, even when $n=1$ and the utility function is linear (see examples in  \cite{neely-renewal-jmlr}).  In general, an optimal solution must learn an optimal \emph{conditional distribution} for choosing a given row $i \in \{1, ..., S[k]\}$ given the observed matrix $M[k]$.  The solution space is enormous. The optimal conditional distribution depends on the distribution $\mu_M$ on the random matrix $M[k]$, where $\mu_M$ describes the distribution on the number of rows $S[k]$  and the corresponding joint distribution across all entries of the $S[k]\times (n+1)$ matrix $M[k]$.    

The number of parameters that describes $\mu_M$ can be very large. Thus, accurate 
estimation of $\mu_M$ is typically impossible to do over a reasonable timescale. Nevertheless, we want an algorithm that does not know the distribution $\mu_M$, but learns to make optimal decisions in a timely manner.  For this, rather than estimate $\mu_M$, we use time-varying weights that evolve according to sampled functionals of the distribution. This is similar to the classic max-weight methodology in \cite{tass-radio-nets}\cite{tass-server-allocation} and the drift-plus-penalty paradigm in \cite{sno-text} which treat various types of opportunistic scheduling problems.   However, the  fractional and multi-objective structure of the current problem requires a new method that we develop in this paper.

We want a control algorithm that approaches the optimal utility $\psi^*$ as quickly as possible \emph{without knowing the probability distribution} for the random matrices. 
 Further, we want the algorithm to be \emph{adaptive} so that it learns to make new decisions when probability distributions change. 
For this, we modify the assumption that $M[k]$ is i.i.d. over all $k \in \{1, 2, 3, ...\}$ to a \emph{block i.i.d.} model where
matrices are i.i.d. over successive blocks of tasks, but the distribution can change arbitrarily from one block to the next. Further, the start and end times of each block are unknown. Specifically, for 
positive integers $m_1, m_2$ with $m_1<m_2$ define 
\begin{align*}
\overline{R}[m_1, m_2] &= \frac{1}{m_2-m_1+1} \sum_{k=m_1}^{m_2} R[k]\\
\overline{T}[m_1, m_2] &= \frac{1}{m_2-m_1+1} \sum_{k=m_1}^{m_2} T[k]
\end{align*}
Fix $\epsilon>0$. We develop an algorithm that can be implemented over all tasks $k \in \{1, 2, 3, ...\}$ and
that has the following property: 
For any block of tasks $\{m_1, ..., m_2\}$ with  size  greater than or equal to $1/\epsilon^2$, our algorithm satisfies 
\begin{equation} \label{eq:adaptation-time} 
\psi\left( \frac{\expect{\overline{R}[m_1, m_2]}}{\expect{\overline{T}[m_1,m_2]}}\right) \geq \psi^* - O(\epsilon)
\end{equation} 
regardless of the probability law for $M[k]$ for $k \in \{1, ..., m_1-1\}$. Thus, the \emph{adaptation time} is $O(1/\epsilon^2)$. 

It is known that achieving $O(1/\epsilon^2)$ \emph{adaptation time} is more difficult than achieving $O(1/\epsilon^2)$ \emph{convergence time}, where convergence time only requires the guarantee \eqref{eq:adaptation-time} to be 
achieved by a block that starts at task $1$, so $m_1=1$. This is much easier because the block starts at a known time, so algorithms can be configured for this known starting time.  For example, a Robbins-Monro algorithm for a similar renewal optimization problem in \cite{neely-renewal-jmlr} uses a \emph{diminishing stepsize} parameter that can be initialized at task $1$. In general, diminishing stepsize algorithms cannot be used for the block i.i.d. model because the start times of new blocks are unknown, so it is unclear when to reset the stepsize parameter. 

The $O(1/\epsilon^2)$ adaptation time that we achieve in this paper is \emph{optimal}. Indeed, prior work in \cite{neely-renewal-jmlr} proves a converse bound on convergence time of $\Omega(1/\epsilon^2)$ for the simpler problem of maximizing the ratio of averages $\expect{\overline{R}_1[m]}/\expect{\overline{T}[m]}$
for \emph{one renewal process} $(R_1[k], T[k])$ for $k \in \{1, 2, 3, ...\}$. In contrast, our current paper 
treats a vector of  $n$ renewal processes, together with a concave utility. Our method extends the recent work \cite{neely-adaptive-renewals} for constrained optimization of a single renewal process to the multi-reward case with a concave utility function. 

\subsection{Prior work} 

Renewal optimization problems of this type are introduced in
\cite{renewal-opt-tac} (see also \cite{sno-text}). The algorithm in
\cite{renewal-opt-tac} uses a drift-plus-penalty ratio method that makes a decision every new task that maximizes a ratio  of expectations $\expect{\Delta[k] +vp[k]}/\expect{T[k]}$, where $\Delta[k]$ represents the change in a Lyapunov function, $v$ is a positive weight, $p[k]$ is the penalty of interest for task $k$, and $T[k]$ is the task duration.  Direct implementation of the algorithm requires knowledge of the probability distribution. A heuristic method that does not need the probability distribution is also given in \cite{renewal-opt-tac}: That method approximates the optimal ratio of averages using a window of previous samples. 

Work in \cite{neely-renewal-jmlr} considers the simpler problem of 
maximizing  $\lim_{m\rightarrow\infty}\expect{\overline{R}_1[m]}/\expect{\overline{T}_1[m]}$ in a system with just one renewal process.  The converse bound on convergence time of $\Omega(1/\epsilon^2)$ is proven there. That prior work also provides a Robbins-Monro algorithm with a diminishing stepsize that achieves the $\Theta(1/\epsilon^2)$ bound. 
However, diminishing stepsize methods make increasing investments in a \emph{single} probability model and cannot adapt if probability distributions change.  The diminishing stepsize method of \cite{neely-renewal-jmlr} can be shown to have \emph{infinite adaptation time}. 

The current paper is most similar to our 
recent work  \cite{neely-adaptive-renewals} which considers a renewal problem of maximizing a time average reward subject to time average constraints. That work achieves an optimal $\Theta(1/\epsilon^2)$ adaptation time, but  does not have a concave utility function.  The current paper leverages the algorithm of \cite{neely-adaptive-renewals}, together with a simple virtual queue transformation,  to achieve an optimal adaptation time of $\Theta(1/\epsilon^2)$ for the utility-based renewal problem. The virtual queues in the current paper have the desirable property of being naturally bounded.  In contrast, the virtual queues 
of \cite{neely-adaptive-renewals} are not naturally bounded and a more complicated analysis based on truncating the queues is needed there. 

The fractional nature of our problem is similar to a \emph{linear fractional program}  \cite{fox-linear-fractional-mdp}\cite{boyd-convex}. Linear fractional programs can be transformed into convex programs via a nonlinear change of variables. 
Work in \cite{fox-linear-fractional-mdp} uses such a nonlinear transformation for offline construction of a Markov decision policy. The current paper treats online optimization of time averages. We cannot use the transformation methods of \cite{fox-linear-fractional-mdp}\cite{boyd-convex} because time averages are not preserved under nonlinear transformations. 

This paper leverages the powerful Lyapunov drift method developed for control of queueing networks in \cite{tass-radio-nets}\cite{tass-server-allocation}. Related opportunistic scheduling methods for fixed timeslot systems (without the renewal structure) are treated via stochastic optimimization 
\cite{shroff-opportunistic}, stochastic primal-dual  \cite{vijay-allerton02}\cite{prop-fair-down}\cite{atilla-primal-dual-jsac}\cite{stolyar-greedy}, and drift-plus-penalty \cite{neely-fairness-ton}\cite{neely-energy-it}\cite{sno-text}. Work in \cite{atilla-renewal-2019}  treats a different type of renewal optimization problem that relates to choosing arms in a multi-armed bandit system with rewards that arise after a random duration of time.  The adaptation concept considered in this paper is similar in spirit to the type of block guarantees available through fixed-stepsize online convex optimization \cite{online-convex}. Adaptive algorithms are also of interest in estimation \cite{robbins-monro-fixed-stepsize-sigmetrics} (where an algorithm that dynamically uses two fixed stepsizes is considered) and in convex bandit problems \cite{Luo2022AdaptiveBC} and \cite{haipeng-bandit1}.  Bandit problems and online convex optimization problems have a structure that is different from the opportunistic scheduling problem of the current paper.

\subsection{Notation} 

Vectors $x \in \mathbb{R}^n$ are treated as column vectors, so the transpose $x^{\top}$ is a row vector. We use $\norm{x}$ to denote the standard Euclidean norm: 
$$ \norm{x} = \sqrt{x^{\top}x}=\sqrt{\sum_{i=1}^n x_i^2}$$ 

\section{Preliminaries} 

\subsection{Variation on Jensen's inequality} 

Fix $n \in \mathbb{N}$. Let $D$ be a convex subset of $\mathbb{R}^n$.   Fix a probability space $(\Omega, \mathcal{F}, P)$. Let $X:\Omega \rightarrow\mathbb{R}^n$ be a random vector and suppose 
$$ X(\omega) \in D \quad \forall \omega \in \Omega$$
 Let $T:\Omega\rightarrow (0, \infty)$ be a positive-valued random variable (possibly correlated with $X$).  
The following variation on Jensen's inequality 
 is given in \cite{sno-text}. We present a different proof below. 

\begin{lem}  Let $\psi:D\rightarrow\mathbb{R}$ be a concave function. Suppose the expectations $\expect{T}$, $\expect{TX}$, $\expect{T\psi(X)}$  are finite (note that $0<\expect{T}<\infty$).  Then  
$$\frac{\expect{T\psi(X)}}{\expect{T}} \leq \psi\left(\frac{\expect{TX}}{\expect{T}} \right)$$
\end{lem}

\begin{proof} 
Define a new probability measure $\tilde{P}:\mathcal{F}\rightarrow\mathbb{R}$ by 
$$ \tilde{P}[A] = \frac{\expect{T1_A}}{\expect{T}} \quad \forall A \in \mathcal{F}$$
where $1_A$ is an indicator function that is $1$ if event $A$ is true, and zero else. It can be shown that $\tilde{P}$ satisfies the three axioms of probability.  Let $\tilde{\mathbb{E}}$ denote an expectation on the probability space $(\Omega, \script{F}, \tilde{P})$. For any  
nonnegative random variable $Y:\Omega \rightarrow [0, \infty)$ we have 
\begin{align*}
\tilde{\mathbb{E}}\left[Y\right] &= \int_0^{\infty} \tilde{P}\left[Y>t\right]dt\\
&= \int_0^{\infty} \frac{\expect{T1_{\{Y>t\}}}}{\expect{T}}dt\\
&\overset{(a)}{=} \frac{1}{\expect{T}} \expect{T \int_0^{\infty} 1_{\{Y>t\}}dt} \\
&\overset{(b)}{=}\frac{\expect{TY}}{\expect{T}} 
\end{align*}
where (a) uses Fubini-Tonelli to exchange the order of the integral and expectation; (b) uses
$\int_0^{\infty} 1_{\{Y>t\}}dt =Y$. By writing the random variable $\psi(X)$ in terms of its positive and negative parts, and doing the same for the components of random vector $X$,  it follows that 
\begin{align}
\tilde{\mathbb{E}}\left[\psi(X)\right] &= \frac{\expect{T\psi(X)}}{\expect{T}} \label{eq:jensen1-sub}\\
\tilde{\mathbb{E}}\left[X\right] &= \frac{\expect{TX}}{\expect{T}}\label{eq:jensen2-sub} 
\end{align}
On the other hand, by the standard Jensen's inequality we have 
$$ \tilde{\mathbb{E}}\left[\psi(X)\right] \leq \psi\left(\tilde{\mathbb{E}}\left[X\right]\right)$$
Substituting equalities \eqref{eq:jensen1-sub}-\eqref{eq:jensen2-sub} 
into the above inequality proves the result. 
\end{proof} 

\subsection{Pushback}

\begin{lem}\label{lem:pushback} (Pushback lemma) Let $X\subseteq\mathbb{R}^n$ be a convex set. Let $f:X\rightarrow\mathbb{R}$ be a convex function. Fix $\alpha>0$, $v\in\mathbb{R}^n$. Suppose $x^*$ minimizes $f(x) + \alpha\norm{x-v}^2$ over all $x \in X$. Then
\begin{equation} \label{eq:pushback} 
f(x^*) + \alpha\norm{x^*-v}^2\leq f(x) + \alpha\norm{x-v}^2-\alpha\norm{x-x^*}^2 \quad \forall x\in X
\end{equation} 
\end{lem} 

We call the final term $-\alpha\norm{x-x^*}^2$ on the right-hand-side of \eqref{eq:pushback} a \emph{pushback term} because it strengthens the inequality in comparison to the basic inequality that would be obtained by definition of $x^*$ being a minimizer.  Lemma \ref{lem:pushback} is well known.  It can be proven either by observing that $f(x) + \alpha\norm{x-v}^2$ is $(2\alpha)$-strongly convex, or by viewing $\alpha\norm{x-v}^2$ as a Bregman divergence and using the Bregman divergence inequalities from Lemma 3.2 in \cite{bregman-pushback}.  Bregman divergence is developed in \cite{bregman-original}.

\section{Problem formulation} 

\subsection{Boundedness assumptions} 

Assume there are  positive constants $t_{min}, t_{max}$ that act as lower and upper bounds
on the task duration, so for all tasks $k \in \{1, 2, 3, ...\}$ and all rows $i \in \{1, ..., S[k]\}$ we have 
\begin{equation} \label{eq:bounded1} 
0<t_{min} \leq T_i[k] \leq t_{max}
\end{equation} 
Further suppose the reward vector options are always in a bounded box. Specifically, define 
\begin{align*}
r_{min} &= (r_{1,min}, ..., r_{n,min})\\
r_{max}&= (r_{1,max}, ..., r_{n,max})
\end{align*}
Assume for all $k \in \{1, 2, 3, ...\}$ and all $i \in \{1, ..., S[k]\}$ that 
\begin{equation} \label{eq:bounded2}
(R_{i,1}[k], ..., R_{i,n}[k]) \in [r_{min}, r_{max}]
\end{equation} 
It follows that under all possible algorithms for making decisions over time, we have for all $k \in \{1, 2, 3, ...\}$: 
\begin{align}
T[k] \in [t_{min}, t_{max}]\label{eq:Tbound}\\
R[k] \in [r_{min}, r_{max}]  \label{eq:Rbound}
\end{align}
For $i \in \{1, ..., n\}$ define
\begin{align*}
\theta_{i,min} &= \min\left\{\frac{r_{i,min}}{t_{min}}, \frac{r_{i,min}}{t_{max}}  \right\} \\
\theta_{i,max} &=  \max\left\{\frac{r_{i,max}}{t_{min}}, \frac{r_{i,max}}{t_{max}}  \right\} 
\end{align*}
where we recall that rewards can be negative, such as when they are used to model penalties like cost or energy. 
Define 
\begin{align}
\theta_{min} &= (\theta_{1,min}, ..., \theta_{n,min}) \label{eq:theta-min} \\
\theta_{max} &= (\theta_{1,max}, ..., \theta_{n,max}) \label{eq:theta-max}\\
[\theta_{min}, \theta_{max}] &= \{(x_1, ..., x_n) \in \mathbb{R}^n: \theta_{i,min}\leq x_i\leq \theta_{i,max}\} \label{eq:theta-cube}
\end{align}
It follows that for all possible decision rules and all tasks $k$ we have 
$$ R[k]/T[k] \in [\theta_{min}, \theta_{max}]$$

\subsection{Utility function} 

Note that $[\theta_{min}, \theta_{max}]\subseteq\mathbb{R}^n$. 
Let $\psi:[\theta_{min}, \theta_{max}]\rightarrow\mathbb{R}$ be a concave, continuous, and entrywise nondecreasing function. 
We call $\psi$ a \emph{utility function}.  It follows that 
$$ \psi(\theta_{min}) \leq \psi(\theta) \leq \psi(\theta_{max}) \quad \forall \theta \in [\theta_{min}, \theta_{max}]$$
The function $\psi$ is not necessarily differentiable. However, we assume for each $\theta \in [\theta_{min}, \theta_{max}]$ that there is 
at least one subgradient vector $\psi'(\theta) \in \partial \psi(\theta)$ that satisfies the following two properties: 
\begin{enumerate} 
\item Subgradient inequality: For all $\theta, u \in [\theta_{min}, \theta_{max}]$ we have 
\begin{equation} \label{eq:bounded-subgradient1}
 \psi(u) \leq \psi(\theta) + \psi'(\theta)^{\top}(u-\theta) 
 \end{equation} 
\item Bounded subgradients: 
There is a vector $\beta=(\beta_1, ..., \beta_n)$ with $\beta_i\geq 0$ for all $i$ such that 
\begin{equation} \label{eq:bounded-subgradient2} 
 0\leq \psi'(\theta) \leq \beta \quad \forall \theta \in [\theta_{min}, \theta_{max}]
 \end{equation} 
\end{enumerate}

\subsection{Stochastic model} \label{section:stochastic-model} 

Assume task matrices $M[k]$ satisfy the boundedness assumptions \eqref{eq:bounded1}-\eqref{eq:bounded2}. For simplicity of exposition, assume the matrices $\{M[k]\}_{k=1}^{\infty}$ are i.i.d. over tasks $k \in \{1, 2, 3, ...\}$ with some distribution $\mu_M$. (This assumption is eventually modified to assume i.i.d. behavior occurs only over one block of tasks.)  The distribution $\mu_M$ determines the distribution on the random number of rows $S[k] \in \mathbb{N}$ and also the joint distribution across all random real entries of the $S[k]\times (n+1)$ matrix $M[k]$. 

Let $\script{P}$ denote the set of all conditional probability distributions for choosing $row[k] \in \{1, ..., S[k]\}$ given the observed matrix $M[k]$ for task $k$. Given $\pi \in \script{P}$, let $(T^{\pi}[k], R^{\pi}[k])$ denote the random duration and reward vector for task $k$, under the assumption that $row[k] \in \{1, ..., S[k]\}$ is chosen randomly according to $\pi$. It is assumed that the randomness used to carry out the selection of $row[k]$  is independent of all previous decisions. Formally, we can assume existence of i.i.d. random variables $\{U[k]\}_{k=1}^{\infty}$ that are uniformly distributed over $[0,1]$, where $U[k]$ is used as an external source of randomness to implement the selection of $row[k]$ according to conditional distribution $\pi$.

Define $\Gamma$ as the set of all one-shot expectations of $(T[k], R[k])$ that can be achieved for task $k$: 
$$ \Gamma = \{\expect{(T^{\pi}[k], R^{\pi}[k])} \in \mathbb{R}^{n+1} : \pi  \in \script{P}\}$$
The boundedness assumptions \eqref{eq:bounded1}-\eqref{eq:bounded2} ensure $\Gamma$ is bounded.  It can be shown that $\Gamma$ is also closed, so it is a compact subset of $\mathbb{R}^{n+1}$. 
Now consider the following deterministic problem: 
\begin{align}
\mbox{Maximize:}  \quad & \psi(\theta) \label{eq:det1}\\
\mbox{Subject to:} \quad & \theta = r/t \label{eq:det2} \\
& (t,r) \in \Gamma \label{eq:det3} 
\end{align} 
The constraint \eqref{eq:det2} avoids division by zero because $t\geq t_{min}>0$ for all $(r,t) \in \Gamma$.  The problem \eqref{eq:det1}-\eqref{eq:det3} is equivalent to maximizing the continuous function $\psi(r/t)$ over the compact set $(t,r) \in \Gamma$, so there is a maximizer $(t^*, r^*) \in \Gamma$.  Define 
$$\theta^* = r^*/t^* $$
It can be shown that $\psi^*=\psi(\theta^*)$, where $\psi^*$ is the optimal objective in 
the dynamic problem \eqref{eq:p1}-\eqref{eq:p2} (see related proofs in \cite{neely-renewal-jmlr}\cite{sno-text}).  It follows that there is a particular $\pi^* \in \script{P}$, corresponding to a conditional probability rule for observing $M[k]$ and choosing a potentially randomized 
$(T^*[k], R^*[k])$, such that 
$$ \expect{(T^*[k], R^*[k])} = (t^*, r^*)$$
Define the \emph{history} $H[k]$ up to but not including task $k>1$: 
$$ H[k]=(M[1], M[2], ..., M[k-1], A[1], ..., A[k-1])$$
where $A[r]$ denotes a vector of decisions (or \emph{actions}) 
made at step $r \in \{0, ..., k-1\}$ (the decision vector $A[r]$ is formally defined in equation \eqref{eq:Ak} of 
the next section).  Define $H[1]=0$. Since $M[k]$ is independent of $H[k]$ we obtain 
\begin{equation} \label{eq:stationary-randomized} 
\expect{(T^*[k], R^*[k]) \vert H[k]} = (t^*, r^*)
\end{equation}

In the infinite horizon version of the problem, the controller does not know the distribution $\mu_M$ of the i.i.d. matrices $\{M[k]\}_{k=1}^{\infty}$. 
It must learn to make good row-selection decisions so that the problem  \eqref{eq:p1}-\eqref{eq:p2} is solved. That is, we want the achieved utility to converge to the optimal utility $\psi^*$ associated with the (unknown) distribution $\mu_M$.   Alternatively, we may be satisfied with an algorithm  that achieves an infinite horizon utility  within $\epsilon$ of $\psi^*$, for some suitably small $\epsilon>0$. 

This paper focuses attention on a related problem where performance analysis is done only over a \emph{finite block of tasks} $\{m_1, ..., m_2\}$  for which matrices are i.i.d. with some unknown distribution $\mu_M$. 
Fix $\epsilon>0$. We explore the fundamental \emph{adaptation time} required to achieve utility over the block that is within $O(\epsilon)$ 
of optimal utility $\psi^*$, where $\psi^*$ is defined as the optimal utility we could achieve in an imaginary infinite horizon problem with i.i.d. matrices of distribution $\mu_M$ over the entire infinite horizon. This problem is challenging for several reasons: 
\begin{enumerate} 
\item The infinite horizon target $\psi^*$ is difficult to achieve over only a finite block of tasks, particularly when the distribution $\mu_M$ is unknown. 
\item The sample path and probability law for $M[k]$, for $k\in \{1, ..., m_1-1\}$, is arbitrary. Thus, any learning that takes place before the block of interest may be irrelevant to the current block.
\item  The starting task $m_1$  of our block of interest is potentially unknown, so we cannot ``reset'' the algorithm to treat task $m_1$ as if it is task $1$.  
\end{enumerate} 

We want an algorithm that provides the $O(\epsilon)$ optimality guarantee of \eqref{eq:adaptation-time} whenever the block size $m_2-m_1+1$ is sufficiently large. The work \cite{neely-renewal-jmlr} shows that it is generally impossible to achieve \eqref{eq:adaptation-time} unless $m_2-m_1+1\geq \Omega(1/\epsilon^2)$, even if we assume $n=1$ (so there is only one reward) and $m_1=1$ (so our block starts at task $1$). 
We develop an algorithm that does not know the probability distributions, but achieves this bound over any block of $\Theta(1/\epsilon^2)$ tasks, assuming matrices $M[k]$ are i.i.d. over that block.  This is useful under a \emph{block-i.i.d. model} where successive blocks of tasks have i.i.d. distributions that can change arbitrarily from block to block. Under this model, if all blocks have size at least $1/\epsilon^2$, then our single algorithm achieves utility over each block that is within $O(\epsilon)$ of the optimal utility for the (unknown) distribution for that block. 

\section{Algorithm} 

\subsection{Auxiliary variables} 

Define the following positive constants: 
\begin{align*}
\gamma_{min} &= 1/t_{max} \\
\gamma_{max} &= 1/t_{min} 
\end{align*}
For $k \in \{1, 2, 3, ...\}$ and $i \in \{1, ..., n\}$ 
define the following additional decision variables: 
\begin{align}
&\gamma[k] \in [\gamma_{min}, \gamma_{max}]\label{eq:gamma}\\
&\theta_i[k] \in [\theta_{i,min}, \theta_{i,max}]\label{eq:theta} 
 \end{align}
 Define $\theta[k]=(\theta_1[k], ..., \theta_n[k])$. Hence
 $$ \theta[k] \in [\theta_{min}, \theta_{max}] \quad \forall k \in \{1, 2, 3, ...\}$$
 We call these new decision variables $(\gamma[k], \theta[k])$ \emph{auxiliary variables}.  At the start of every frame $k$, they can be chosen as any real number in the intervals specified in \eqref{eq:gamma}-\eqref{eq:theta}. 
Intuitively, the value $\gamma[k]$ can be viewed as a guess of $1/t^*$, 
while $\theta_i[k]$ is a guess of $\theta_i^*=r_i^*/t^*$, where $(t^*, r^*)$ is defined in Section \ref{section:stochastic-model}.  We shall find it is important to ensure the auxiliary variables do not dramatically change
from one task to the next. However, for the purposes of adapting to a change in the probability distribution of $M[k]$, there needs to be enough flexibility so the auxiliary variables  can move towards 
a new optimality point. 

%Intuitively, we want the guesses to get close to the true values they are approximating.  We do not want them to vary much over time. However, they should have enough flexibility to adapt when probability distributions change. 

 \subsection{Virtual queues} 
 
 Following the \emph{virtual queue} paradigm of \cite{sno-text}, we design real-valued processes $\{J[k]\}_{k=1}^{\infty}$ and $\{Q_i[k]\}_{k=1}^{\infty}$ for $i \in \{1, ..., n\}$ with initial conditions $J[1]=0$, $Q_i[1]=0$ for $i \in \{1, ..., n\}$, and update equation
 \begin{align}
 J[k+1] &= J[k]+(T[k]-1/\gamma[k]) \quad \forall k \in \{1, 2, 3, ...\} \label{eq:j-update}\\
 Q_i[k+1] &= \max\{Q_i[k] + (\theta_i[k]T[k]-R_i[k]), 0\}  \quad \forall k \in \{1, 2, 3, ...\} \label{eq:q-update} 
 \end{align}
 where $(T[k], R[k]) = (T[k], R_1[k], ..., R_n[k])$ is the row of $M[k]$ that is chosen for task $k$, and $\gamma[k], \theta_1[k], ..., \theta_n[k]$ are the auxiliary variables chosen for task $k$.  As described in \cite{sno-text}, if we control the system to keep $J[k]$ bounded, then the average of the actual task durations $T[k]$ (averaged over a suitably large block of tasks) is close to the average of $1/\gamma[k]$ over that same block.  Likewise, if we keep $Q_i[k]$ bounded, then the average of $\theta_i[k]T[k]$ over a large block of tasks is less than or equal to the average of $R_i[k]$ over the same block (plus a small error that depends on the block size). 
 The specific virtual queue $J[k]$ defined above is the same as the one used for the renewal problem in  \cite{neely-adaptive-renewals}. 
 
 \subsection{Lyapunov drift} 
 
 For $k \in \{1, 2, 3, ...\}$ define 
 \begin{align*} 
 Q[k]&=(Q_1[k], ..., Q_n[k]) \\
 L[k]&= \frac{1}{2}\left[J[k]^2 + \norm{Q[k]}^2\right]\\
 \Delta[k] &= L[k+1]-L[k]
 \end{align*}
 where $\norm{Q[k]}^2=\sum_{i=1}^nQ_i[k]^2$.   The value $L[k]$
 is a scalar measure of the size of the virtual queue vector $(J[k], Q[k])$ 
 at step $k$. We call $L[k]$ a \emph{Lyapunov function}. We call $\Delta[k]$ the \emph{Lyapunov drift}. [Traditionally, the term \emph{Lyapunov drift} is associated with a conditional expectation of $\Delta[k]$ given the system state at step $k$. For simplicity we refer to $\Delta[k]$ itself 
 as the drift, although we do eventually take a conditional expectation.]
  
  To ensure all virtual queues are bounded, 
 it is enough to ensure $L[k]$ is bounded.  For this, we choose decision variables at the start of each 
 task $k$ to make $\Delta[k]$ small (preferably nonpositive). 
 Of course, we also need to make decisions that provide high rewards. 
 It is important to strike a balance between the drift objective and the reward objective. At a high level, we make use of the \emph{drift-plus-penalty} paradigm of \cite{sno-text}.  However, to accommodate the 
 challenging fractional structure of this renewal problem, we develop
 a novel \emph{layered algorithm} that is different from the classical Lyapunov drift framework of \cite{tass-radio-nets}\cite{tass-server-allocation} and also different from  drift-plus-penalty methods \cite{sno-text}\cite{neely-fairness-ton}, stochastic opportunistic scheduling \cite{shroff-opportunistic},  and stochastic primal-dual   \cite{vijay-allerton02}\cite{prop-fair-down}\cite{atilla-primal-dual-jsac}\cite{stolyar-greedy}.   The algorithm is closest in concept to our recent work \cite{neely-adaptive-renewals} that treats a similar renewal problem without the concave objective function. 
 
 Before describing the algorithm, we provide the following bound on $\Delta[k]$.  For each $k \in \{1, 2, 3, ...\}$ define the \emph{decision vector} $A[k]$ (also called an \emph{action vector})  by 
\begin{equation} \label{eq:Ak}
A[k]=(T[k], R[k], \gamma[k], \theta[k])
\end{equation} 
 
 \begin{lem} Consider any sequence of decisions $\{A[k]\}_{k=1}^{\infty}$ that satisfy \eqref{eq:Ak}, \eqref{eq:Tbound}-\eqref{eq:Rbound}, and \eqref{eq:gamma}-\eqref{eq:theta}.  Then for each $k \in \{1, 2, 3, ...\}$ we have 
 \begin{equation} \label{eq:Delta}
 \Delta[k]\leq b + J[k](T[k]-1/\gamma[k]) + Q[k]^{\top}(\theta[k]T[k]-R[k])
 \end{equation} 
 where 
 \begin{equation} \label{eq:b} 
 b = \frac{(t_{max}-t_{min})^2}{2} + \frac{1}{2}\sum_{i=1}^n(r_{i,max}-r_{i,min})^2
 \end{equation} 
 \end{lem} 
 
 \begin{proof} 
 Fix $k \in \{1, 2, 3, ...\}$. Squaring \eqref{eq:j-update} gives
\begin{equation} \label{eq:proof-delta1} 
J[k+1]^2 = J[k]^2 + (T[k]-1/\gamma[k])^2 + 2J[k](T[k]-1/\gamma[k])
\end{equation} 
 Fix $i \in \{1, ..., n\}$. Squaring \eqref{eq:q-update} and using $\max\{u,0\}^2\leq u^2$ for all $u\in\mathbb{R}$ gives
\begin{equation}
Q_i[k+1]^2\leq Q_i[k]^2 + (\theta_i[k]T[k]-R_i[k])^2 + 2Q_i[k](\theta_i[k]T[k]-R_i[k])\label{eq:proof-delta2}
\end{equation} 
Summing up \eqref{eq:proof-delta1}-\eqref{eq:proof-delta2} and dividing by $2$ gives 
$$\Delta[k]\leq b[k] + J[k](T[k]-1/\gamma[k]) +Q[k]^{\top}(\theta[k]T[k]-R[k])$$
where we define
$$ b[k] = \frac{(T[k]-1/\gamma[k])^2}{2} + \frac{1}{2}\sum_{i=1}^k(\theta_i[k]T[k]-R_i[k])^2$$
The value $b$ in \eqref{eq:b} is a worst-case upper bound on $b[k]$.
 \end{proof} 
 
 \subsection{Drift-plus-penalty bound} 
 
 To simplify the layered analysis, for each $k \in \{1, 2, 3, ...\}$ define 
 \begin{align}
 c[k] &= J[k]-\frac{1}{\epsilon}\psi(\theta[k]) +Q[k]^{\top}\theta[k] \label{eq:ck}\\
 d[k] &= c[k]T[k]- Q[k]^{\top}R[k] \label{eq:dk} 
 \end{align}
 \begin{lem} Consider any sequence of decisions $\{A[k]\}_{k=1}^{\infty}$ that satisfy \eqref{eq:Ak}, \eqref{eq:Tbound}-\eqref{eq:Rbound}, and \eqref{eq:gamma}-\eqref{eq:theta}.  Then for each $k \in \{1, 2, 3, ...\}$ we have 
\begin{equation} \label{eq:layer-lem}
\gamma[k]\left[\Delta[k]-\frac{T[k]}{\epsilon}\psi(\theta[k])\right] \leq  b\gamma[k]- J[k] + \gamma[k]d[k]  
\end{equation} 
where $c[k]$ and $d[k]$ are defined in \eqref{eq:ck}-\eqref{eq:dk}. The term $\Delta[k] - \frac{T[k]}{\epsilon} \psi(\theta[k])$ shall be called the \emph{drift-plus-penalty} term.
 \end{lem} 
 \begin{proof} 
 Adding $-\frac{T[k]}{\epsilon}\psi(\theta[k])$ to both sides of \eqref{eq:Delta} gives 
 $$ \Delta[k]-\frac{T[k]}{\epsilon}\psi(\theta[k])\leq b + J[k](T[k]-1/\gamma[k]) + Q[k]^{\top}(\theta[k]T[k]-R[k]) -\frac{T[k]}{\epsilon}\psi(\theta[k])$$
 Multiplying both sides by the nonnegative value $\gamma[k]$ gives
 \begin{align*}
 &\gamma[k]\left[\Delta[k]-\frac{T[k]}{\epsilon}\psi(\theta[k])\right] \\
 &\leq b\gamma[k]- J[k] +\gamma[k]\left[J[k]T[k]+Q[k]^{\top}(\theta[k]T[k]-R[k]) - \frac{T[k]}{\epsilon}\psi(\theta[k]) \right]\\
 &\overset{(a)}{=} b\gamma[k]- J[k] +\gamma[k]\left[c[k]T[k]-Q[k]^{\top}R[k] \right]\\
 &\overset{(b)}{=} b\gamma[k]- J[k] +\gamma[k]d[k]
 \end{align*}
 where (a) substitutes the definition of $c[k]$ from \eqref{eq:ck}; (b) substitutes the definition of $d[k]$ from \eqref{eq:dk}. 
 \end{proof} 
  
 \subsection{Layered algorithm} 
 
 We provide the following \emph{layered algorithm}, similar to the \emph{hierarchical optimization} technique of \cite{neely-adaptive-renewals}.  The algorithm uses positive parameters: 
 $$\epsilon>0 \quad, \quad \alpha>0$$
 These parameters are right-sized later. 
 For initial conditions, define 
 \begin{align}
 J[1]&=0 \label{eq:init1}\\
 Q[1]&= (0, ..., 0)\label{eq:init2} \\
 \gamma[0] &= 1/t_{min} \label{eq:init3}
 \end{align}

Fix $\epsilon>0$. At the start of each task $k \in \{1, 2, 3, ...\}$, observe values of virtual queues $J[k], Q[k]$ and treat them as known constants. Also observe $(\gamma[k-1], \theta[k-1])$ and treat them as known constants.  Recall that $[\theta_{min}, \theta_{max}]\subseteq\mathbb{R}^n$ is the multidimensional box defined by \eqref{eq:theta-cube}. Then:
\begin{itemize} 
\item Layer 1: Choose $\theta[k] \in [\theta_{min}, \theta_{max}]$ to maximize
$$ \frac{1}{\epsilon} \psi(\theta[k]) - Q[k]^{\top}\theta[k]$$
This maximizes a continuous and concave function over a compact box, so at least one maximizer $\theta[k]$ exists. 

\item Layer 2: Observe the value of $\theta[k]$ from the Layer 1 decision, which determines the value of $c[k]$ in \eqref{eq:ck}. Observe the new task matrix $M[k]$ and choose a row $(T[k], R[k]) \in Row(M[k])$ to minimize 
$$ c[k]T[k]  - Q[k]^{\top}R[k]$$
In other words, for each row $i$ of $M[k]$ (each row being a vector in $\mathbb{R}^{n+1}$), we take an inner product of that row with the vector $(c[k], -Q[k]) \in \mathbb{R}^{n+1}$. Then we choose the row with the smallest inner product.  

\item Layer 3: Observe the value of $c[k]$ defined at Layer 2, and observe the values of $T[k], R[k]$ from the Layer 2 decision. The values $c[k], T[k], R[k]$ together determine the value of $d[k]$ in \eqref{eq:dk}. Then choose $\gamma[k]\in [\gamma_{min}, \gamma_{max}]$ to minimize 
\begin{equation} \label{eq:layer3} 
\gamma[k]d[k]+ \frac{(\alpha/\epsilon) \gamma[k-1]}{2}(\gamma[k]-\gamma[k-1])^2
\end{equation} 
The simple closed-form solution is  
\begin{equation} \label{eq:layer-3-result}
\gamma[k] = \left[ \gamma[k-1] - \frac{d[k]}{(\alpha/\epsilon)\gamma[k-1]} \right]_{\gamma_{min}}^{\gamma_{max}}
\end{equation} 
[Observe that $\gamma[k-1]\geq \gamma_{min}>0$, so \eqref{eq:layer-3-result} avoids division by $0$.] 
\end{itemize} 

At the end of each task $k$ the virtual queues are updated via \eqref{eq:j-update}-\eqref{eq:q-update}. 

\subsection{Discussion of the steps} 

The layered structure of this procedure is analogous to a multi-timescale algorithm. However, the decisions are all made for one timescale. The three layered steps are performed one after the other, all at the start of task $k$. The Layer 1 decision chooses $\theta[k] \in [\theta_{min}, \theta_{max}]\subseteq\mathbb{R}^n$ as the solution to a (possibly nonsmooth) convex optimization problem. When the concave utility function $\psi$ is separable,  the $\theta[k]$ value can often be solved in closed form. For example, given weights $w_i>0, \beta>0$ we can consider the example utility function 
$$ \psi(\theta)= \sum_{i=1}^k w_i\log(1+\beta \theta_i) \quad \forall \theta \in [\theta_{min}, \theta_{max}]$$
This models a weighted form of proportional fairness (see \cite{kelly-charging}\cite{kelly-shadowprice} for such utility functions). Then
\begin{align*}
\theta_i[k] &= \left\{\begin{array}{cc}
\theta_{i,max} & \mbox{ if $Q_i[k]=0$} \\
 \left[\frac{w_i}{\epsilon Q_i[k]} - \frac{1}{\beta}\right]_{\theta_{i,min}}^{\theta_{i,max}} & \mbox{ if $Q_i[k]>0$} 
 \end{array}\right.
 \end{align*}

The Layer 2 decision takes $S[k]$ inner products, where $S[k]$ is the random number of rows of matrix $M[k]$.  Suppose $s_{max}$ is an upper bound on the number of rows of $M[k]$. Each inner product is between two vectors in $\mathbb{R}^{n+1}$, so the worst-case complexity of the Layer 2 decision for task $k$ is $(n+1)s_{max}$ (measured in multiply-adds).  If each task has at most $10$ processing options, then $s_{max}=10$ and the worst-case complexity is $10(n+1)$, which is typically easy (for example, when $n\leq 10$).  On the other hand, if the task matrices can have as many 
as $10^6$ rows, then $s_{max}=10^6$ and the complexity of Layer 2 can be large. Fortunately, the inner products can be computed in a parallel fashion where different batches of rows are given to different processors that assist in 
the computation. Finally, the Layer 3 decision always has the simple closed form result \eqref{eq:layer-3-result}. 

The usage of the term $(\gamma[k]-\gamma[k-1])^2$ in the Layer 3 decision is similar to prox-type regularizers for convex optimization. However, that term is introduced here only to penalize variation in the $\gamma[k]$ decision from one task to the next.  It is not standard to weight this prox-type term by the $\gamma[k-1]$ value as we do. This is done for a specific reason that is made apparent in the analysis of the algorithm. 

\subsection{Key inequalities} 

Fix $k \in \{1, 2, 3, ...\}$. Assume the matrix $M[k]$ arises from some distribution associated with values 
\begin{equation} \label{eq:distribution-star}
(t^*, r^*), \quad \theta^*=r^*/t^*, \quad \psi^*=\psi(\theta^*)
\end{equation} 
as defined in Section \ref{section:stochastic-model}, and $M[k]$ is independent of $H[k]$. 

\begin{enumerate} 
\item Since the  Layer 1 decision chooses $\theta[k]$ to maximize $(1/\epsilon)\psi(\theta[k]) - Q[k]^{\top}\theta[k]$ 
over all possible $\theta \in [\theta_{min}, \theta_{max}]$, and since $\theta^*\in [\theta_{min}, \theta_{max}]$, we have 
\begin{equation} \label{eq:layer1-key}
-(1/\epsilon)\psi(\theta[k]) + Q[k]^{\top}\theta[k] \leq -(1/\epsilon)\psi^* + Q[k]^{\top}\theta^*
\end{equation} 
\item Since the Layer 2 decision chooses row $(T[k], R[k])$ to minimize 
$c[k]T[k]- Q[k]^{\top}R[k]$ over all alternative row decisions, including the randomized decisions $(T^*[k], R^*[k])$ defined in Section \ref{section:stochastic-model}, we have 
\begin{equation} \label{eq:layer2-key}
c[k]T[k]- Q[k]^{\top}R[k] \leq c[k]T^*[k]- Q[k]^{\top}R^*[k]
\end{equation} 
\item Since the Layer 3 decision chooses $\gamma[k] \in [\gamma_{min}, \gamma_{max}]$ to minimize the strongly convex function \eqref{eq:layer3}, we have by the pushback lemma:
\begin{align*} 
&\gamma[k] d[k] + \frac{(\alpha/\epsilon)\gamma[k-1]}{2}(\gamma[k]-\gamma[k-1])^2\nonumber\\
&\leq \gamma d[k] + \frac{(\alpha/\epsilon) \gamma[k-1]}{2}(\gamma-\gamma[k-1])^2 - \frac{(\alpha/\epsilon)\gamma[k-1]}{2}(\gamma-\gamma[k])^2  \quad \forall \gamma \in [\gamma_{min}, \gamma_{max}]
\end{align*} 
Since $1/t^*\in [\gamma_{min}, \gamma_{max}]$ we have from the above inequality  
\begin{align} 
&\gamma[k]d[k] + \frac{(\alpha/\epsilon)\gamma[k-1]}{2}(\gamma[k]-\gamma[k-1])^2\nonumber\\
&\leq (1/t^*) d[k] + \frac{(\alpha/\epsilon) \gamma[k-1]}{2}\left[(1/t^*-\gamma[k-1])^2 -(1/t^*-\gamma[k])^2\right]\label{eq:layer3-key}
\end{align} 
\end{enumerate} 

The above inequalities are the key to the first theorem in the next section. 

\section{Performance analysis}

\begin{thm} Fix parameters $\epsilon>0, \alpha>0$.  Fix initial conditions 
according to \eqref{eq:init1}-\eqref{eq:init3}. 
Suppose the algorithm of the previous section 
is run over all tasks $k \in \{1, 2, 3, ...\}$. 
Fix a particular $k \in \{1, 2, 3, ...\}$ and assume $M[k]$ is independent of $H[k]$, and $M[k]$ has distribution associated with values $t^*, r^*,  \theta^*=r^*/t^*, \psi^*$  in \eqref{eq:distribution-star}. Then for this task $k$: 
\begin{align}
 &\expect{\epsilon\Delta[k]-T[k]\psi(\theta[k])}\nonumber \\
&\leq \epsilon b - \psi^*\expect{\frac{1}{\gamma[k]}} + \frac{\alpha}{2}\expect{(1/t^*-\gamma[k-1])^2-(1/t^*-\gamma[k])^2}+\frac{1}{2\alpha}\expect{w[k]^2}\label{eq:DPP} 
\end{align}
where $w[k]$ is defined
\begin{equation} \label{eq:wk} 
w[k]=\left[\frac{-\epsilon b}{\gamma[k-1]} + \frac{\psi^*}{\gamma[k]\gamma[k-1]} + \frac{\epsilon \Delta[k] - T[k]\psi(\theta[k])}{\gamma[k-1]} \right]
\end{equation} 
\end{thm} 

\begin{proof} 
Adding the expression $\frac{(\alpha/\epsilon)\gamma[k-1]}{2}(\gamma[k]-\gamma[k-1])^2$ to both sides of \eqref{eq:layer-lem} gives 
\begin{align*}
&\gamma[k]\left[\Delta[k]-\frac{T[k]}{\epsilon}\psi(\theta[k])\right] + \frac{(\alpha/\epsilon)\gamma[k-1]}{2}(\gamma[k]-\gamma[k-1])^2 \\
&\leq  b\gamma[k]- J[k] + \gamma[k]d[k]  + \frac{(\alpha/\epsilon)\gamma[k-1]}{2}(\gamma[k]-\gamma[k-1])^2\\
&\overset{(a)}{\leq}  b\gamma[k]- J[k] + (1/t^*)d[k]  + \frac{(\alpha/\epsilon)\gamma[k-1]}{2}\left[(1/t^*-\gamma[k-1])^2-(1/t^*-\gamma[k])^2\right]\\
&\overset{(b)}{=} b\gamma[k]- J[k] + (1/t^*)\left[ c[k]T[k]- Q[k]^{\top}R[k] \right]  \\
&\quad + \frac{(\alpha/\epsilon)\gamma[k-1]}{2}\left[(1/t^*-\gamma[k-1])^2-(1/t^*-\gamma[k])^2\right]\\
&\overset{(c)}{=} b\gamma[k]- J[k] + (1/t^*)\left[ c[k]T^*[k]- Q[k]^{\top}R^*[k] \right]  \\
&\quad + \frac{(\alpha/\epsilon)\gamma[k-1]}{2}\left[(1/t^*-\gamma[k-1])^2-(1/t^*-\gamma[k])^2\right]
\end{align*}
where (a) holds by \eqref{eq:layer3-key}; (b) substitutes the definition of $d[k]$ from \eqref{eq:dk}; (c) holds by \eqref{eq:layer2-key}.  By 
substituting the definition of  $c[k]$ from \eqref{eq:ck} into the final expression above we obtain 
\begin{align*}
&\gamma[k]\left[\Delta[k]-\frac{T[k]}{\epsilon}\psi(\theta[k])\right] + \frac{(\alpha/\epsilon)\gamma[k-1]}{2}(\gamma[k]-\gamma[k-1])^2 \\
&\leq  b\gamma[k]- J[k] + (1/t^*)\left[\left(J[k]-\frac{1}{\epsilon}\psi(\theta[k]) +Q[k]^{\top}\theta[k] \right)T^*[k]- Q[k]^{\top}R^*[k] \right]  \\
&\quad + \frac{(\alpha/\epsilon)\gamma[k-1]}{2}\left[(1/t^*-\gamma[k-1])^2-(1/t^*-\gamma[k])^2\right]\\
&\leq  b\gamma[k]- J[k] + (1/t^*)\left[\left(J[k]-\frac{1}{\epsilon}\psi^* +Q[k]^{\top}\theta^* \right)T^*[k]- Q[k]^{\top}R^*[k] \right]  \\
&\quad + \frac{(\alpha/\epsilon)\gamma[k-1]}{2}\left[(1/t^*-\gamma[k-1])^2-(1/t^*-\gamma[k])^2\right]
\end{align*}
where the final inequality uses \eqref{eq:layer1-key} and the fact $(1/t^*)T^*[k]\geq 0$.  Taking conditional expectations given $H[k]$, observing that 
$\gamma[k-1]$ is $H[k]$-measurable (as are the virtual queues $Q[k], J[k]$)  and using the fact 
$$\expect{(T^*[k], R^*[k]) \vert H[k]}=(t^*,r^*)$$ 
from \eqref{eq:stationary-randomized} gives 
\begin{align*}
&\expect{\gamma[k]\left[\Delta[k]-\frac{T[k]}{\epsilon}\psi(\theta[k])\right] + \frac{(\alpha/\epsilon)\gamma[k-1]}{2}(\gamma[k]-\gamma[k-1])^2 \vert H[k]} \\
&\leq b\expect{\gamma[k]\vert H[k]} - J[k] + (1/t^*)\left[\left(J[k]-\frac{1}{\epsilon}\psi^* +Q[k]^{\top}\theta^* \right)t^*- Q[k]^{\top}r^* \right]  \\
&\quad + \frac{(\alpha/\epsilon)\gamma[k-1]}{2}\expect{(1/t^*-\gamma[k-1])^2-(1/t^*-\gamma[k])^2 \vert H[k]}\\
&= b\expect{\gamma[k]\vert H[k]} - \frac{1}{\epsilon}\psi^* + \frac{(\alpha/\epsilon)\gamma[k-1]}{2}\expect{(1/t^*-\gamma[k-1])^2-(1/t^*-\gamma[k])^2 \vert H[k]}
\end{align*}
where the final equality uses $r^*/t^*=\theta^*$.  Dividing both sides by $\gamma[k-1]$ gives
\begin{align*}
&\expect{\frac{\gamma[k]}{\gamma[k-1]}\left[\Delta[k]-\frac{T[k]}{\epsilon}\psi(\theta[k])\right] + \frac{(\alpha/\epsilon)}{2}(\gamma[k]-\gamma[k-1])^2 \vert H[k]} \\
&\leq b\expect{\frac{\gamma[k]}{\gamma[k-1]}\vert H[k]} - \frac{\psi^*}{\epsilon \gamma[k-1]} + \frac{(\alpha/\epsilon)}{2}\expect{(1/t^*-\gamma[k-1])^2-(1/t^*-\gamma[k])^2 \vert H[k]}
\end{align*}
We have divided by $\gamma[k-1]$, rather than $\gamma[k]$, because $\gamma[k-1]$ is $H[k]$-measurable. Thus, we can freely pull $\gamma[k-1]$ out of the conditional expectations. We have premultiplied the prox term by $\gamma[k-1]$ in anticipation of dividing by $\gamma[k-1]$. 
The above inequality is the heart of the theorem. The rest is simply rearranging terms. 

Multiplying both sides by $\epsilon$ gives 
\begin{align*}
&\expect{\frac{\gamma[k]}{\gamma[k-1]}\left[\epsilon\Delta[k]-T[k]\psi(\theta[k])\right] + \frac{\alpha}{2}(\gamma[k]-\gamma[k-1])^2 \vert H[k]} \\
&\leq \epsilon b\expect{\frac{\gamma[k]}{\gamma[k-1]}\vert H[k]} - \frac{\psi^*}{\gamma[k-1]} + \frac{\alpha}{2}\expect{(1/t^*-\gamma[k-1])^2-(1/t^*-\gamma[k])^2 \vert H[k]}
\end{align*}
Rearranging terms by adding and subtracting gives
\begin{align*}
&\expect{\epsilon\Delta[k]-T[k]\psi(\theta[k])  \vert H[k]} \\
&\leq \epsilon b - \psi^*\expect{\frac{1}{\gamma[k]}\vert H[k]} + \frac{\alpha}{2}\expect{(1/t^*-\gamma[k-1])^2-(1/t^*-\gamma[k])^2 \vert H[k]}\\
&\quad +\epsilon b\expect{\frac{\gamma[k]-\gamma[k-1]}{\gamma[k-1]}\vert H[k]} \\ 
&\quad + \psi^*\expect{\frac{1}{\gamma[k]} - \frac{1}{\gamma[k-1]}\vert H[k]} \\
&\quad + \expect{\frac{(\gamma[k-1]-\gamma[k])}{\gamma[k-1]}\left[\epsilon \Delta[k] - T[k]\psi(\theta[k])\right] \vert H[k]} \\
&\quad -\expect{\frac{\alpha}{2}(\gamma[k-1]-\gamma[k])^2 \vert H[k]} 
\end{align*}
Define 
\begin{align}
z[k] &= \epsilon b\frac{(\gamma[k]-\gamma[k-1])}{\gamma[k-1]} + \psi^*\left(\frac{1}{\gamma[k]} - \frac{1}{\gamma[k-1]}\right)\nonumber\\
&\quad + \frac{(\gamma[k-1]-\gamma[k])}{\gamma[k-1]}\left[\epsilon \Delta[k] - T[k]\psi(\theta[k])\right]\nonumber -\frac{\alpha}{2}(\gamma[k-1]-\gamma[k])^2 \nonumber
\end{align} 
Then 
\begin{align*}
&\expect{\epsilon\Delta[k]-T[k]\psi(\theta[k])  \vert H[k]} \\
&\leq \epsilon b - \psi^*\expect{\frac{1}{\gamma[k]}\vert H[k]} + \frac{\alpha}{2}\expect{(1/t^*-\gamma[k-1])^2-(1/t^*-\gamma[k])^2 \vert H[k]}+\expect{z[k]\vert H[k]} 
\end{align*}
Now 
\begin{align*}
z[k] &= (\gamma[k-1]-\gamma[k]) \underbrace{\left[\frac{-\epsilon b}{\gamma[k-1]} + \frac{\psi^*}{\gamma[k]\gamma[k-1]} + \frac{\epsilon \Delta[k] - T[k]\psi(\theta[k])}{\gamma[k-1]} \right]}_{w[k]} - \frac{\alpha}{2}(\gamma[k-1]-\gamma[k])^2\\
&\leq \frac{w[k]^2}{2\alpha}
\end{align*}
where the final inequality holds by completing the square: 
$$ xw[k] - \frac{\alpha}{2}x^2 \leq \frac{w[k]^2}{2\alpha} \quad \forall x \in \mathbb{R}$$
Thus
\begin{align*}
 &\expect{\epsilon\Delta[k]-T[k]\psi(\theta[k])  \vert H[k]} \\
&\leq \epsilon b - \psi^*\expect{\frac{1}{\gamma[k]}\vert H[k]} + \frac{\alpha}{2}\expect{(1/t^*-\gamma[k-1])^2-(1/t^*-\gamma[k])^2 \vert H[k]}-\frac{1}{2\alpha}\expect{w[k]^2 \vert H[k]} 
\end{align*}
Taking expectations and using the law of iterated expectations gives the result. 
\end{proof}

\subsection{Performance over a block of tasks} 

Now suppose we have a block of tasks $\{k_0, ..., k_0+m-1\}$ for some 
positive integers $k_0, m$, where $k_0$ is the starting task and $m$ is the number of tasks in the block. For simplicity of notation, define 
\begin{align*}
T_{av}&= \frac{1}{m}\sum_{k=k_0}^{k_0+m-1}T[k]\\
R_{av} &= \frac{1}{m}\sum_{k=k_0}^{k_0+m-1}R[k]
\end{align*}

\begin{thm} Fix parameters $\epsilon>0$ and $\alpha = 1/\epsilon$.  Fix initial conditions 
according to \eqref{eq:init1}-\eqref{eq:init3}. 
Suppose the algorithm of the previous section 
is run over all tasks $k \in \{1, 2, 3, ...\}$. Suppose that task matrices $M[k]$ are i.i.d. over a block $\{k_0, ..., k_0+m-1\}$, with distribution during this block that is associated with values $t^*, r^*, \theta^*=r^*/t^*, \psi^*$  in \eqref{eq:distribution-star}. Then 
\begin{align*}
\psi\left(\frac{\expect{R_{av}}}{\expect{T_{av}}}\right)- \psi^* &\geq \frac{ -\frac{\epsilon \expect{L[k_0]}}{m} - \epsilon b - \frac{\alpha}{2m}\expect{(1/t^*-\gamma[k_0+m-1])^2}  - \frac{w_{max}^2}{2\alpha}  - \frac{|\psi^*|}{m}\expect{|J[k_0+m]-J[k_0]|}}{t_{min}}\\
&\quad  - \frac{\frac{1}{m}\expect{\beta^{\top}Q[k_0+m]}}{t_{min}}
\end{align*}
where $\beta$ is defined in \eqref{eq:bounded-subgradient2};  $w_{max}\geq 0$ is a worst-case bound on $|w[k]|$ to be computed in the next subsection, where $w[k]$ is defined in \eqref{eq:wk}.  When $m\geq 1/\epsilon^2$, the next subsection shows  the absolute value of the error terms on the right-hand-side of the above inequality is at most  $\Theta(\epsilon)$, regardless of $k_0$ and regardless of the probability law or sample path associated with $\{M[k]\}$ for $k<k_0$. Thus 
$$ \psi\left(\frac{\expect{R_{av}}}{\expect{T_{av}}}\right)- \psi^* \geq -\Theta(\epsilon) \quad \mbox{when $m\geq 1/\epsilon^2$}$$
so the asymptotic adaptation time is  $\Theta(1/\epsilon^2)$, which is optimal. 
\end{thm} 

\begin{proof} 
From \eqref{eq:DPP}, we have for all $k \in \{k_0, ..., k_0+m-1\}$: 
\begin{align*}
 &\expect{\epsilon\Delta[k]-T[k]\psi(\theta[k])}\nonumber \\
&\leq \epsilon b - \psi^*\expect{\frac{1}{\gamma[k]}} + \frac{\alpha}{2}\expect{(1/t^*-\gamma[k-1])^2-(1/t^*-\gamma[k])^2}+\frac{1}{2\alpha}w_{max}^2
\end{align*}
Summing over $k \in \{k_0, ..., k_0+m-1\}$ gives
\begin{align*}
&\epsilon\expect{L[k_0+m]-L[k_0]} - \sum_{k=k_0}^{k_0+m-1}\expect{T[k]\psi(\theta[k])} \\
&\leq \epsilon bm - \psi^*\sum_{k=k_0}^{k_0+m-1} \expect{\frac{1}{\gamma[k]}} + \frac{\alpha}{2}\expect{(1/t^*-\gamma[k_0-1])^2 -(1/t^*-\gamma[k_0+m-1])^2} + \frac{mw_{max}^2}{2\alpha}
\end{align*}
Dividing by $m$ and rearranging terms gives 
\begin{align*}
&\frac{1}{m}\sum_{k=k_0}^{k_0+m-1} \expect{T[k]\psi(\theta[k])} - \psi^*\cdot \left(\frac{1}{m}\sum_{k=k_0}^{k_0+m-1}\expect{T[k]}\right) \\
&\geq -\epsilon b + \frac{\epsilon}{m}\expect{L[k_0+m]-L[k_0]} \\
&\quad - \psi^*\cdot \frac{1}{m}\sum_{k=k_0}^{k_0+m-1}\expect{T[k]-1/\gamma[k]} \\
&\quad - \frac{\alpha}{2m}\expect{(1/t^*-\gamma[k_0-1])^2 -(1/t^*-\gamma[k_0+m-1])^2}  - \frac{w_{max}^2}{2\alpha} \\
&\overset{(a)}{=}  -\epsilon b + \frac{\epsilon}{m}\expect{L[k_0+m]-L[k_0]} \\
&\quad - \psi^*\cdot \frac{1}{m}\expect{J[k_0+m]-J[k_0]} \\
&\quad - \frac{\alpha}{2m}\expect{(1/t^*-\gamma[k_0-1])^2 -(1/t^*-\gamma[k_0+m-1])^2}  - \frac{w_{max}^2}{2\alpha} \\
&\geq  -\epsilon b - \frac{\epsilon}{m}\expect{L[k_0]} \\
&\quad -|\psi^*|\cdot \frac{1}{m}\expect{|J[k_0+m]-J[k_0]|} \\
&\quad - \frac{\alpha}{2m}\expect{(1/t^*-\gamma[k_0-1])^2}  - \frac{w_{max}^2}{2\alpha} \\
\end{align*}
where equality (a) uses the virtual queue update equation \eqref{eq:j-update}. 
Dividing both sides by $\frac{1}{m}\sum_{k=k_0}^{k_0+m-1}\expect{T[k]}$ and using the fact that this expression is at most $t_{min}$ 
gives 
\begin{align*}
&\frac{\frac{1}{m}\sum_{k=k_0}^{k_0+m-1} \expect{T[k]\psi(\theta[k])}}{\frac{1}{m}\sum_{k=k_0}^{k_0+m-1}\expect{T[k]}}  - \psi^* \\
&\geq \frac{  -\epsilon b - \frac{\epsilon}{m}\expect{L[k_0]} -|\psi^*|\cdot \frac{1}{m}\expect{|J[k_0+m]-J[k_0]|}  - \frac{\alpha}{2m}\expect{(1/t^*-\gamma[k_0-1])^2}  - \frac{w_{max}^2}{2\alpha}}{t_{min}}
\end{align*}
Using the variation on Jensen's inequality for the concave function $\psi$ gives: 
\begin{equation} \label{eq:temp-variation}
\frac{\frac{1}{m}\sum_{k=k_0}^{k_0+m-1} \expect{T[k]\psi(\theta[k])}}{\frac{1}{m}\sum_{k=k_0}^{k_0+m-1}\expect{T[k]}}  \leq \psi\left(\frac{\frac{1}{m}\sum_{k=k_0}^{k_0+m-1} \expect{T[k]\theta[k]}}{\frac{1}{m}\sum_{k=k_0}^{k_0+m-1}\expect{T[k]}}\right) = \psi\left(\frac{g}{\expect{T_{av}}}\right) 
\end{equation} 
where we define the constant $g$ by 
$$ g = \frac{1}{m}\sum_{k=k_0}^{k_0+m-1} \expect{T[k]\theta[k]}$$
However, from \eqref{eq:q-update} we obtain for all $i \in \{1, ..., n\}$
$$ Q_i[k+1]\geq Q_i[k] + \theta_i[k]T[k]-R_i[k]$$ 
and in vector form this gives
\begin{equation} \label{eq:vector-form}
Q[k+1]-Q[k] \geq \theta[k]T[k]-R[k]
\end{equation} 
Thus
\begin{align*}
g&=\frac{1}{m}\sum_{k=k_0}^{k_0+m-1} \expect{T[k]\theta[k]} \\
&=\frac{1}{m}\sum_{k=k_0}^{k_0+m-1} \expect{R[k] + (\theta[k]T[k]-R[k])}\\
&\overset{(a)}{\leq} \frac{1}{m}\sum_{k=k_0}^{k_0+m-1} \expect{R[k] + (Q[k+1]-Q[k])}\\
&= \expect{R_{av}} + \frac{1}{m}\expect{Q[k_0+m]-Q[k_0]}\\
&\leq \expect{R_{av}} + h
\end{align*}
where (a) holds by \eqref{eq:vector-form}; we define the nonnegative constant $h$ by 
$$ h =  \frac{1}{m}\expect{Q[k_0+m]}$$
 Thus, since $\psi$ is entrywise nondecreasing, 
\begin{align*}
\psi\left(\frac{g}{\expect{T_{av}}}\right) &\leq \psi\left(\frac{\expect{R_{av}}}{\expect{T_{av}}} +\frac{h}{\expect{T_{av}}}  \right)\\
&\leq  \psi\left(\frac{\expect{R_{av}}}{\expect{T_{av}}}\right) + \beta^{\top}\frac{h}{\expect{T_{av}}}
\end{align*}
where we have used \eqref{eq:bounded-subgradient1}-\eqref{eq:bounded-subgradient2}. Substituting this into \eqref{eq:temp-variation} proves the result. 
\end{proof} 

\subsection{Bounding the error terms} 

An argument similar to that of \cite{hao-harvesting-ton2019} 
can show that all virtual queues are bounded by a constant of size $\theta(1/\epsilon)$ for all $k \in \{1, 2, 3, ...\}$  (proof omitted for brevity). 
Thus, for $m\geq 1/\epsilon^2$, $\alpha = 1/\epsilon$, and all $k_0$ we have 
\begin{align*}
&\frac{\epsilon}{m} L[k_0] \leq \epsilon^3 L[k_0] \leq \epsilon^3\Theta(\frac{1}{\epsilon^2}) = \Theta(\epsilon) \\
&\frac{1}{m}|J[k_0+m]-J[k_0]| \leq \epsilon^2\Theta(1/\epsilon) = \Theta(\epsilon) \\
&\frac{\alpha}{m} \leq (1/\epsilon)/\epsilon^2 = \epsilon \\
&\frac{1}{m}\beta^{\top}Q[k_0+m] \leq \epsilon^2\Theta(1/\epsilon) = \Theta(\epsilon) \\
&\frac{w[k]^2}{2\alpha} = \frac{\epsilon w[k]^2}{2}
\end{align*}
It suffices to prove $|w[k]|\leq w_{max}= \Theta(1)$.  By definition of $w[k]$ in \eqref{eq:wk}, it is clear we simply need to prove that
$$ |\epsilon \Delta[k] | \leq \Theta(1) \quad \forall k $$
We have 
\begin{align*} 
|\Delta[k]| &= |L[k+1]-L[k]| \\
&= \frac{1}{2}\left|J[k+1]^2-J[k]^2 + \norm{Q[k+1]}^2-\norm{Q[k]}^2\right|\\
&= \frac{1}{2}\left|\underbrace{(J[k+1]-J[k])}_{O(1)}\underbrace{(J[k+1]+J[k])}_{O(1/\epsilon)} + \underbrace{(Q[k+1]-Q[k])^{\top}}_{O(1)}\underbrace{(Q[k+1]+Q[k])}_{O(1/\epsilon)}\right|\\
&\leq \Theta(1/\epsilon)
\end{align*}
and so $|\epsilon \Delta[k]|\leq \Theta(1)$. 

\section{Conclusion} 

This work considers utility maximization in a multi-objective renewal optimization problem. The problem involves fractional objectives inside a concave utility function. The goal is to develop an online algorithm that does not know the probability distribution, but converges quickly to optimal utility. 
It achieves an adaptation time of $O(1/\epsilon^2)$, which matches a known converse bound from prior work \cite{neely-renewal-jmlr}. Hence, it provides optimal adaptation time.  
The algorithm uses a novel layered optimization framework that enables fast adaptation.  This can be viewed 
as a utility-based extension 
of our recent work \cite{neely-adaptive-renewals} which treats minimization of one renewal process subject to time average constraints.  The utility-based extension of the current paper yields naturally bounded virtual queues (without the truncation procedure of  \cite{neely-adaptive-renewals}), which simplifies the adaptation time analysis. 

 % ------------------------------------------------------------------------
%GATHER{Xbib.bib}   % For Gather Purpose Only
%GATHER{Paper.bbl}  % For Gather Purpose Only
\bibliographystyle{unsrt}
\bibliography{../../../latex-mit/bibliography/refs}
\end{document}